\documentclass[10pt, draft, reqno]{amsart}

\newtheorem{theorem}{Theorem}[section]
\newtheorem{lemma}[theorem]{Lemma}

\newtheorem{assumption}[theorem]{Assumption}

\newtheorem{definition}[theorem]{Definition}
\newtheorem{example}[theorem]{Example}
\newtheorem{remark}[theorem]{Remark}

\newcommand{\Tr}{\mathop{\mathrm{Tr}}}
\renewcommand{\d}{\/\mathrm{d}\/}
\def\s{^{\star}}
\def\st{\tilde{\sigma}}
\def\u{u^{n, \varepsilon}}
\def\ue{u^{\varepsilon}}
\def\uv{u^{\varepsilon}_{v}}
\def\n{u_{v_n}}
\def\m{w_{v_n}}
\def\uve{u^{\varepsilon}_{{v}^{\varepsilon}}}
\def\w{w^{\varepsilon}_{{v}^{\varepsilon}}}
\def\ve{v^{\varepsilon}}
\def\we{w^{\varepsilon}}
\def\e{\varepsilon}
\def\sni{\smallskip\noindent}
\begin{document}

\title{Large Deviations for the Stochastic Shell Model of Turbulence}

\author[U. Manna, S. S. Sritharan and P. Sundar]
{U. Manna\\ Max Planck Institute for Mathematics in the Sciences\\ Leipzig, Germany\\manna@mis.mpg.de
 \\ \\
S.S. Sritharan\\ Department of Mathematics\\ University of Wyoming\\Laramie, USA\\sri@uwyo.edu \\ \\
and\\ \\ P. Sundar \\Department of Mathematics\\ Louisiana State University\\
Baton Rouge, USA\\sundar@math.lsu.edu}

\maketitle
\begin{abstract}
In this work we first prove the existence and uniqueness of a strong solution to stochastic GOY model of turbulence with a small multiplicative noise.
 Then using the weak convergence approach, Laplace principle for solutions of the stochastic GOY model is established in certain Polish space.
 Thus a Wentzell-Freidlin type large deviation principle is established utilizing certain results by Varadhan and Bryc.
\end{abstract}
\sni Subject class[2000]: Primary 60F10; Secondary 60H15, 76D03,
76D06.

\sni Keywords: GOY model, Large deviations, Local monotonicity.

\section{Introduction}
The large deviations theory is among the most classical areas in
probability theory with many deep developments and applications (see
Dembo and Zeitouni \cite{DZ}, Deuschel and Stroock \cite{DeS},
Dupuis and Ellis \cite{DE}, Freidlin and Wentzell \cite{FW}, Stroock
\cite{St}, Varadhan \cite{Va}). Although it appears to be no
literature on the large deviation results for stochastic shell model
of turbulence, a few authors have proved the Wentzell-Freidlin type
large deviations for the two dimensional stochastic Navier-Stokes
equations with additive noise (e.g. Chang \cite{Ch1}) and also for
multiplicative noise (e.g. Sritharan and Sundar \cite{SS2}). For
Donsker-Varadhan type large deviation study related to Navier-Stokes
equations we refer the readers to Quastel and Yau\cite{Qa}. For the treatment related to stochastic two-dimensional vorticity equations see the work of Amirdjanova and Xiong \cite{Am}. Several authors have established the Wentzell-Freidlin type large deviation
estimates for a class of infinite dimensional stochastic
differential equations (see for eg., Budhiraja and Dupuis
\cite{BD1}, Chow \cite{Ch2}, Da Prato and Zabczyk \cite{DaZ},
Kallianpur and Xiong \cite{KX}, Sowers \cite{So}). In these works
the proofs of large deviation principle (LDP) (see Definition
\ref{LDP} below) usually rely on first approximating the original
problem by time-discretization so that LDP can be shown for the
resulting simpler problems via contraction principle, and then
showing that LDP holds in the limit. The discretization method to
establish LDP was introduced by Wentzell and Freidlin\cite{FW}.

Dupuis and Ellis \cite{DE} have combined weak convergence methods to
the stochastic control approach developed earlier by Fleming
\cite{Fl} to the large deviations theory. Our work is based on the
theory introduced by Budhiraja and Dupuis \cite{BD1}, where they
used the stochastic control and weak convergence approach to obtain
the LDP for the family $\{\mathcal{G}^{\e}(W(\cdot))\}_{\e >0}$,
where $\mathcal{G}^{\e}$ is an appropriate family of measurable maps
from the Wiener space to some Polish space. Their work relied on the
fact that the LDP is equivalent to Laplace principle (see Definition
\ref{LP1} below) if the underlying space is Polish, which is in turn
a consequence of Varadhan's Lemma ( see Lemma \ref{VL} below) and
Bryc's converse to Varadhan's Lemma (see Lemma \ref{BL} below). We
refer the reader to Theorems 1.2.1 and 1.2.3 in Dupuis and Ellis
\cite{DE}.

In the next Section, we give some definitions and basic properties
from the large deviation theory. In later part of this section, we
describe briefly the work of Budhiraja and Dupuis \cite{BD1} to set
up the ground for our main work. In Section 3, we formulate the
abstract stochastic GOY model when the noise coefficient is small.
We then prove certain a priori energy estimates with exponential
weight. These estimates together with the local monotonicity
property of the sum of the linear and non linear operators play a
fundamental role to prove the existence and uniqueness of the strong
solution. In the last Section, we establish the LDP for the
stochastic GOY model perturbed by a small multiplicative noise.

\section{Large Deviation Principle}
\setcounter{equation}{0} In this section we will give an abstract
formulation and basic properties for a class of large deviation
problems. Let us denote by $X$ a complete separable metric space and
$\{P_{\varepsilon}: \varepsilon > 0\}$ a family of probability
measures on the Borel subsets of $X$.
\begin{definition}
A function $I : X\rightarrow [0, \infty]$ is called a \emph{rate function} if $I$ is lower semicontinuous.
A rate function $I$ is called a \emph{good rate function} if for arbitrary $M \in [0, \infty)$, the level set $K_M = \{x: I(x)\leq M\}$
is compact in $X$
\end{definition}
\begin{definition}\label{LDP}(Large Deviation Principle)
 We say that a family of probability measures $\{P_{\varepsilon}\}$ satisfies the \emph{large deviation principle} (LDP) with a good rate
 function $I$ satisfying,
\begin{enumerate}
\item[(i)] for each closed set $F\subset X$
$$ \limsup_{\varepsilon\rightarrow 0}\ \varepsilon\ \log P_{\varepsilon}(F) \leq -\inf_{x\in F} I(x),$$
\item[(ii)] for each open set $G\subset X$
$$ \liminf_{\varepsilon\rightarrow 0}\ \varepsilon\ \log P_{\varepsilon}(G) \geq -\inf_{x\in G} I(x).$$
\end{enumerate}
\end{definition}
\begin{remark}
For any given $\{P_{\varepsilon}: \varepsilon > 0\}$ there is at most one rate function governing the large deviations of $\{P_{\varepsilon}: \varepsilon > 0\}$.
\end{remark}
\begin{example}
As a simplest example we recall the one dimensional version of Cram\'{e}r's theorem~\cite{Va}. Let $X_n = (Y_1 + Y_2 + \cdots + Y_n)/n$ where
 $\{Y_j\}$ are independent random variables with a common distribution $\alpha$. Assume that the moment generating function
 $$M(\theta) = E[\exp(\theta Y)] = \int e^{\theta y} \d \alpha(y)$$ is finite for all $\theta$. Then $\{X_n\}$ satisfies the LDP with rate
 function (see, Deuschel and Stroock\cite{DeS})
$$I(x)=\sup_{\theta}[\theta x - \log M(\theta)].$$
\end{example}

\begin{example}
We choose Schilder's theorem as second example which has many important applications in large deviation theory, such as, in the derivation of the Strassen's renowned Law of Iterated Logarithm, in the Wentzell and Freidlin's estimate on the large deviations of randomly perturbed dynamical systems, to name a few.

Let $d\in\mathbb{Z}^{+}$ and
$$A_0 = \big\{\phi\in \mathcal{C}([0, \infty); \mathbb{R}^d): \phi(0)=0\ \text{and}\ \lim_{t\rightarrow\infty}
\frac{|\phi(t)|}{t}=0\big\}.$$
For $\phi\in A_0$ define
$$\|\phi\|_{A_0} = \sup_{t\geq 0} \frac{|\phi(t)|}{1+t}.$$
Then notice that $A_0$ is a separable real Banach space \cite{DE}.
Next, we define $\mathrm{H}^1 = \mathrm{H}^1([0, \infty);
\mathbb{R}^d)$ to be the space of $\phi\in A_0$ with the property
that $\phi(t) = \int_0^t \dot{\phi}(s)\d s, t\geq 0$, for some
$\dot{\phi} \in \mathrm{L}^2([0, \infty); \mathbb{R}^d)$ and set
$\|\phi\|_{\mathrm{H}^1} = \|\dot{\phi}\|_{\mathrm{L}^2([0, \infty);
\mathbb{R}^d)}$, for $\phi\in\mathrm{H}^1$.

Now for given $T>0$, we define $I_{T}: A_0\rightarrow [0, \infty]$ by
\begin{equation*}
I_T(\psi) =\left\{\begin{aligned}
& \frac{1}{2}\int_0^T |\dot{\psi}(t)|^2 \d t\quad\text{if}\ \psi\in \mathrm{H}^1,\\
&\infty\quad\quad\quad\quad\quad\quad\ \text{otherwise}.
\end{aligned}\right.\end{equation*}

Let $\{W(t)\}$ be a standard Wiener process in $\mathbb{R}^d$. Let
the process $$ X_n(t) = \frac{1}{\sqrt{n}} W(t)$$ takes values in a
Polish space $E$. Then $\{X_n\}$ satisfies the LDP on $E$ with the
rate function $I_T$ (see, Dupuis and Ellis \cite{DE}).
\end{example}

\begin{definition}\label{LP1}(Laplace Principle)
For $h \in C_{b}([0, 1])$,
\begin{equation}
\lim_{n \to \infty}{1\over n}\log\int_0^1\,e^{-nh(x)}\,dx = -\min_{x
\in [0, 1]}\,h(x).\label{Lap}
\end{equation}
\end{definition}
\begin{lemma}\label{VL}(Varadhan's Lemma \cite{Va})
Let $E$ be a Polish space and $\{X^{\varepsilon}: \varepsilon > 0\}$ be a family of $E$-valued random elements satisfying LDP with rate function $I$.
Then $\{X^{\varepsilon}: \varepsilon > 0\}$ is said to satisfy the Laplace principle on $E$ with the same rate function $I$ if for all $h \in C_{b}(E)$,
\begin{equation}
\lim_{\varepsilon \rightarrow 0} {\varepsilon }\log E\big\{\exp\big[- \frac{1}{\varepsilon}h(X^{\varepsilon})\big]\big\} = -\inf_{x \in E}
\{h(x) + I(x)\}. \label{LP}
\end{equation}
\end{lemma}
\begin{lemma}\label{BL}(Bryc's Lemma \cite{DE})
The Laplace principle implies the LDP with the same rate function. More precisely, if $\{X^{\varepsilon}: \varepsilon > 0\}$ satisfies the Laplace principle on the Polish space $E$ with the rate function $I$ and if the limit
\begin{equation*}
\lim_{\varepsilon \rightarrow 0} {\varepsilon }\log E\big\{\exp\big[- \frac{1}{\varepsilon}h(X^{\varepsilon})\big]\big\} = -\inf_{x \in E}
\{h(x) + I(x)\}
\end{equation*}
is valid for all $h \in C_{b}(E)$, then $\{X^{\varepsilon}: \varepsilon > 0\}$ satisfies the LDP on $E$ with rate function $I$.
\end{lemma}
Note that, Varadhan's Lemma together with Bryc's converse of Varadhan's Lemma state that for Polish space valued random elements,
the Laplace principle and the large deviation principle are equivalent.

Let $(\Omega, \mathcal{F}, P)$ be a probability space equipped with
an increasing  family \\
$\{\mathcal{F}_t\}_{0 \leq t \leq T}$ of
sub-sigma-fields of $\mathcal{F}$  satisfying the usual conditions
of right continuity and $P$-completeness. Let $H$ be a real separable Hilbert space and $Q$ be a strictly positive, symmetric, trace class operator on $H$.
\begin{definition}
A stochastic process $\{W(t)\}_{0\leq t\leq T}$ is said to be an $H$-valued $\mathcal{F}_t$-adapted Wiener process with covariance operator $Q$ if
\begin{enumerate}
\item For each non-zero $h\in H$, $|Q^{1/2}h|^{-1} (W(t), h)$ is a standard one-dimensional Wiener process,
\item For any $h\in H, (W(t), h)$ is a martingale adapted to $\mathcal{F}_t$.
\end{enumerate}
\end{definition}
If $W$ is a an $H$-valued Wiener process with covariance operator $Q$ with $\Tr Q < \infty$, then
$W$ is a Gaussian process on $H$ and $$ E(W(t)) = 0,\quad \text{Cov}\ (W(t)) = tQ, \quad t\geq 0.$$
Let $H_0 = Q^{1/2}H.$ Then $H_0$ is a Hilbert space equipped with the inner product $(\cdot, \cdot)_0$,
$$(u, v)_0 = (Q^{-1/2}u, Q^{-1/2}v),\ \forall u, v\in H_0,$$
where $Q^{-1/2}$ is the pseudo-inverse of $Q^{1/2}$. Since $Q$ is a trace class operator, the imbedding of $H_0$ in $H$ is Hilbert-Schmidt.

Let $L_Q$ denote the space of linear operators $S$ such that $S Q^{1/2}$ is a Hilbert-Schmidt operator from $H$ to $H$. Define the norm on the space $\mathrm{L}_Q$ by $|S|_{\mathrm{L}_Q}^2 = \Tr(SQS^*)$.

Let $$\mathcal{A} = \Big\{H_0-\text{valued}\ \{\mathcal{F}_t\}-\text{predictable processes}\ v\ \text{such that}\ \int_0^T |v(s)|_0^2 \d s < \infty\ \text{a.s}\Big\}.$$
Define the set $S_N$ of bounded deterministic controls as,
$$S_N = \Big\{v \in\mathrm{L}^2([0, T]; H_0): \int_0^T |v(s)|_0^2 \d s \leq N\Big\}.$$
The set $S_N$ endowed with the weak topology on $\mathrm{L}^2([0, T]; H_0)$ is a Polish space \cite{DS}.

Define $\mathcal{A}_N$ as the set of bounded stochastic controls by
$$\mathcal{A}_N = \Big\{v\in\mathcal{A}: v(\omega) \in S_N,\text{P-a.s.}\Big\}.$$
Let $E$ denote a Polish space, and for $\e >0$ let $\mathcal{G}^{\e}: C([0, T]; H) \rightarrow E$ be a
measurable map. Define $$X^{\e } = \mathcal{G}^{\e}(W(\cdot)).$$
We are interested in the large deviation principle for $X^{\e }$ as $\e \to 0$.

\begin{assumption}\label{assum}
There exists a measurable map $\mathcal{G}^0: C([0, T]: H) \to E$
such that the following hold:

\item{1.} Let $\{v^{\e}: \e > 0\} \subset \mathcal{A}_M$ for some $M < \infty$.
Let $v^{\e}$ converge in distribution as $S_M$-valued random
elements to $v$. Then $\mathcal{G}^{\e}(W(\cdot) +
\frac{1}{\sqrt{\e}}\int_0^. v^{\e}(s)\d s)$ converges in
distribution to $\mathcal{G}^0(\int_0^.v(s)\d s)$.

\item{2.} For every $M < \infty$, the set $$K_M = \Big\{\mathcal{G}^0\big(\int_0^. v(s)\d s\big): v \in S_M\Big\}$$ is a compact subset of $E$.
\end{assumption}
For each $f \in E$, define
\begin{equation}\label{rate}
I(f) = \inf_{\big\{v \in \mathrm{L}^2([0, T]; H_0): f = \mathcal{G}^0(\int_0^.v(s)\d s)\big\}}
\Big\{{1\over 2}\int_0^T\,|v(s)|_0^2\d s\Big\}
\end{equation}
where infimum over an empty set is taken as $\infty$.

We now state an important result by Budhiraja and Dupuis \cite{BD1}.
\begin{theorem}\label{main}
Let $X^{\e} = \mathcal{G}^{\e }(W(\cdot ))$. If $\{\mathcal{G}^{\e}\}$ satisfies the
Assumption \ref{assum}, then the family $\{X^{\e}: \e > 0\}$ satisfies
the Laplace principle in $E$ with rate function $I$ given by \eqref{rate}.
\end{theorem}

\begin{remark}\label{rem1}
\item{1.} Notice that since the underlying space $E$ is Polish, the family $\{X^{\e}: \e > 0\}$ satisfies
the LDP in $E$ with the same rate function $I$.

\item{2.} Assumption 1. is a statement on the weak convergence of a certain family of random variables and is at the core of weak convergence approach to the study of large deviations. Assumption 2. essentially says that the level sets of the rate function are compact. In the next section, we have proved that there exists a unique strong solution $\ue$ of the stochastic GOY model with small multiplicative noise with values in $X=C([0, T]; H) \cap \mathrm{L}^2(0, T; V)$. Since $X$ is Polish, there exists a
    Borel-measurable function $\mathcal{G}^{\e}: C([0, T]; H) \to X$ such that $\ue(\cdot) =
\mathcal{G}^{\e}(W(\cdot))$ a.s. Our main result is to prove the family $\{\mathcal{G}^{\e}\}$ satisfies the Assumption
\ref{assum} so that Theorem \ref{main} can be invoked to prove the LDP for $\{\ue : \e>0\}$.
\end{remark}

\section{The Stochastic GOY Model of Turbulence}
\setcounter{equation}{0} The GOY model
(Gledger-Ohkitani-Yamada)~\cite{OY} is a particular case of so
called ``Shell model" (see, Frisch \cite{Fu}). This model is the
Navier-Stokes equation written in the Fourier space where the
interaction between different modes is preserved between nearest
modes. To be precise, the GOY model describes a one-dimensional
cascade of energies among an infinite sequence of complex
velocities, $\{u_n(t)\}$, on a one dimensional sequence of wave
numbers
$$k_n = k_0 2^n,\quad k_0 > 0,\ n=1, 2, \ldots$$
where the discrete index $n$ is referred to as the ``shell index". The equations of motion of the stochastic GOY model of turbulence have the form
\begin{align}\label{goy}
\frac{\d u_n}{\d t} + \nu k_n^2 u_n &+ i\big(a k_n u\s_{n+1}u\s_{n+2} + b k_{n-1}u\s_{n-1}u\s_{n+1} + \nonumber\\ & +ck_{n-2} u\s_{n-1}u\s_{n-2}\big)
= f_n + \sigma_n(t, u_n)\frac{\d w_n(t)}{\d t}, \quad\text{for}\ n= 1, 2, \ldots,
\end{align}
along with the boundary conditions
\begin{equation}\label{bc}
u_{-1} = u_0 = 0.
\end{equation}
Here $u\s_n$ denotes the complex conjugate of $u_n$, $\nu > 0$ is
the kinematic viscosity and $f_n$ is the Fourier component of the
forcing. $a, b$ and $c$ are real parameters such that energy
conservation condition $a + b + c =0$ holds (see Kadanoff, Lohse,
Wang, and Benzi\cite{Ka}; Ohkitani and Yamada\cite{OY}). For the
standard model $a=-1, b=1/2$ and $c=1/2$. For each $n$, $w_n$ is one
dimensional Brownian motion and the noise coefficient $\sigma_n$ is
assumed to satisfy the following properties,
\begin{enumerate}
\item[a.1.] For all $t \in [0, T]$, there exists a positive constant $K_1$ such that, $$|\sigma_n(t, u_n)|^2 \leq K_1 k_n^2 |u_n|^2,$$
\item[a.2.]  For all $t \in [0, T]$,  there exists a positive constant $K_2$ such that,
$$|\sigma_n(t, u_n) - \sigma_n(t, v_n)|^2\leq K_2 k_n^2 |u_n -
v_n|^2.$$
\end{enumerate}

\subsection{Functional Setting}
Let $H$ be a real Hilbert space such that
\begin{equation*}
H :=\Big\{u=(u_1, u_2, \ldots) \in \mathbb{C}^{\infty}: \sum_{n=1}^{\infty} |u_n|^2 < \infty\Big\}.
\end{equation*}
For every $u, v \in H$, the scalar product $(\cdot,\cdot)$ and norm $|\cdot|$ are defined on $H$ as
$$(u, v)_H = Re\ \sum_{n=1}^{\infty} u_n v\s_n, \quad |u| = \big(\sum_{n=1}^{\infty} |u_n|^2\big)^{1/2}.$$
Let us now define the space
$$V :=\Big\{u\in H: \sum_{n=1}^{\infty} k_n^2|u_n|^2 < \infty\Big\},$$
which is a Hilbert space equipped with the norm
$$\|u\| = \big(\sum_{n=1}^{\infty} k_n^2|u_n|^2\big)^{1/2}.$$
The linear operator $A: D(A) \rightarrow H$ is a positive definite, self adjoint linear operator defined by
\begin{equation}\label{A}
Au=((Au)_1, (Au)_2, \ldots), \ \text{where}\ (Au)_n = k_n^2 u_n, \quad\forall u\in D(A).
\end{equation}
The domain of $A$, $D(A) \subset H$, is a Hilbert space equipped with the norm
$$\|u\|_{D(A)} = |Au| = \big(\sum_{n=1}^{\infty} k_n^4|u_n|^2\big)^{1/2}, \quad\forall u\in D(A).$$
Since the operator $A$ is positive definite, we can define the power $A^{1/2}$ ,
$$A^{1/2}u = (k_1 u_1, k_2 u_2, \ldots), \quad\forall u=(u_1, u_2, \ldots).$$
Furthermore, we define the space
$$D(A^{1/2}) = \Big\{u=(u_1, u_2, \ldots): \sum_{n=1}^{\infty} k_n^2 |u_n|^2 < \infty\Big\}$$
which is a Hilbert space equipped with the scalar product
$$ (u, v)_{D(A^{1/2})} = (A^{1/2}u, A^{1/2}v), \quad\forall u, v\in D(A^{1/2}),$$
and the norm $$\|u\|_{D(A^{1/2})} = \big(\sum_{n=1}^{\infty} k_n^2 |u_n|^2\big)^{1/2}.$$
Note that $V = D(A^{1/2})$. We consider $V^{\prime} = D(A^{-1/2})$ as the dual space of $V$. Then the following inclusion holds
$$V\subset H = H^{\prime}\subset V^{\prime}.$$
We will now introduce the sequence spaces analogue to Sobolev functional spaces.
For $1\leq p <\infty$ and $s\in\mathbb{R}$
$$\mathrm{W}^{s, p} :=\Big\{u = (u_1, u_2, \ldots): \|A^{s/2}u\|_p = \big(\sum_{n=1}^{\infty}(k_n^s|u_n|)^p\big)^{1/p} < \infty\Big\},$$
and for $p=\infty$
$$\mathrm{W}^{s, \infty} :=\Big\{u = (u_1, u_2, \ldots): \|A^{s/2}u\|_{\infty} = \sup_{1\leq n<\infty} (k_n^s|u_n|) < \infty\Big\},$$
where for $u\in\mathrm{W}^{s, p}$ the norm is defined as
$$\|u\|_{\mathrm{W}^{s, p}} = \|A^{s/2}u\|_p.$$
Here $\|\cdot\|$ denotes the usual norm in the $l^p$ sequence space. It is clear from the above definitions that $W^{1, 2} = V = D(A^{1/2})$.

We now prove a useful Lemma which has been used in this work.
\begin{lemma}\label{La}
For any smooth function $u\in H$, the following holds:
\begin{align}
  \| u \|_{l^4}^4 \leq C |u|^2 \ \|u\|^2.\label{L4}
\end{align}
\end{lemma}

\begin{proof}
Note that,
\begin{align*}
\sum_{n=1}^{\infty} |u_n|^4 &= \sum_{n=1}^{\infty}(k_n |u_n|)\
(k_n^{-1}|u_n|) \ |u_n|^2\\
&\leq (\sup_{1\leq n <\infty} k_n|u_n|)\ \sum_{n=1}^{\infty}(k_n^{-1}|u_n|) \ |u_n|^2\\
&\leq C \|u\|_{W^{1, \infty}}\
\big(\sum_{n=1}^{\infty}(k_n^{-1}|u_n|)^2\big)^{1/2}\
\big(\sum_{n=1}^{\infty} |u_n|^4\big)^{1/2}.
\end{align*}
Thus
\begin{align}\label{L3}
\big(\sum_{n=1}^{\infty} |u_n|^4\big)^{1/2}\leq C \|u\|_{W^{1,
\infty}}\ \|u\|_{V^{\prime}}.
\end{align}
 It is obvious to see that,
$$\|u\|_{W^{1, \infty}} = \| A^{1/2}u\|_{\infty} \leq \|
A^{1/2}u\|_{2} = \|u\|_{V}.$$ Also the inclusion $V\subset H
=H^{\prime}\subset V^{\prime}$ holds. Hence $\|u\|_{V^{\prime}} \leq
C |u|_H$.

So from \eqref{L3}, one can see,
\begin{align*}
\|u\|_{l^4}^2 \leq C |u|_H \ \|u\|_{V}.
\end{align*}
Thus squaring both sides we have the result.
\end{proof}

\subsection{Properties of the linear and nonlinear operators}
We define the bilinear operator $B(\cdot, \cdot): V \times H\rightarrow H$ as
$$B(u, v) = (B_1(u,v), B_2(u, v), \ldots),$$ where
\begin{align*}
B_n(u, v)= ik_n\big(\frac{1}{4}u\s_{n+1} v\s_{n-1} -
\frac{1}{2}(u\s_{n+1} v\s_{n+2} + u\s_{n+2} v\s_{n+1})
 + \frac{1}{8}u\s_{n-1} v\s_{n-2}\big).
\end{align*}
In other words, if $\{e_n\}_{n=1}^{\infty}$ be a orthonormal
basis of $H$, i.e. all the entries of $e_n$ are zero except at the
place $n$ it is equal to $1$, then
\begin{align}\label{B}
B(u, v)= i\sum_{n=1}^{\infty}k_n\big(\frac{1}{4}u\s_{n+1} v\s_{n-1}
- \frac{1}{2}(u\s_{n+1} v\s_{n+2} + u\s_{n+2} v\s_{n+1})
 + \frac{1}{8}u\s_{n-1} v\s_{n-2}\big)e_n.
\end{align}

 The following lemma says that $B(u, v)$ makes
sense as an element of $H$, whenever $u\in V$ and $v\in H$ or $u\in
H$ and $v\in V$. It also says that $B(u, v)$ makes sense as an
element of $V^{\prime}$. Here we state the following lemma which has
been proved in Constantin, Levant and Titi \cite{CLT} for the Sabra
shell model, but one can also prove the similar estimates for the
GOY model (see Barbato, Barsanti, Bessaih, and Flandoli\cite{Ba}).
\begin{lemma}\label{Bprop1}
(i) There exist constants $C_1 >0, C_2 >0$,
\begin{equation}
|B(u, v)| \leq C_1 \|u\| |v|, \quad\forall u\in V, v\in H,
\end{equation}
and
\begin{equation}
|B(u, v)| \leq C_2 |u| \|v\|, \quad\forall u\in H, v\in V.
\end{equation}
(ii) $B: H\times H\rightarrow V^{\prime}$ is a bounded bilinear operator and for a constant $C_3 > 0$
\begin{equation}
\|B(u, v)\|_{V^{\prime}} \leq C_3 |u| |v|, \quad\forall u, v\in H.
\end{equation}
(iii) $B: H\times D(A)\rightarrow V$ is a bounded bilinear operator and for a constant $C_4 > 0$
\begin{equation}
\|B(u, v)\|_{V} \leq C_4 |u| |Av|, \quad\forall u\in H, v\in D(A).
\end{equation}
(iv) For every $u\in V$ and $v\in H$
\begin{equation}\label{1}
(B(u, v), v) = 0.
\end{equation}
\end{lemma}

We now present one more important property of the nonlinear operator $B$ in the following lemma which will play important role in the later part of this section.
\begin{lemma}\label{Bprop2}
If $w=u-v$, then
$$B(u, u)-B(v, v) = B(v, w) + B(w, v) + B(w, w).$$
\end{lemma}
\begin{proof}
The proof is straightforward. We start with the right hand side and express everything in terms of $u$ and $v$ by using $w=u-v$ and rearranging,
\begin{align*}
&B(v, w) + B(w, v) + B(w, w)\\
&=i\sum_{n=1}^{\infty}k_n\Big[\frac{1}{4} w\s_{n-1}v\s_{n+1} - \frac{1}{2} v\s_{n+1}w\s_{n+2} - \frac{1}{2} w\s_{n+1}v\s_{n+2} + \frac{1}{8} v\s_{n-1}w\s_{n-2}+\\
&\quad\quad +\frac{1}{4} v\s_{n-1}w\s_{n+1} - \frac{1}{2} w\s_{n+1}v\s_{n+2} - \frac{1}{2} v\s_{n+1}w\s_{n+2} + \frac{1}{8} w\s_{n-1}v\s_{n-2}+\\
&\quad\quad +\frac{1}{4} w\s_{n-1}w\s_{n+1} - w\s_{n+1}w\s_{n+2} + \frac{1}{8} w\s_{n-1}w\s_{n-2}\Big]e_n\\
&=i\sum_{n=1}^{\infty}k_n\Big[\frac{1}{4} w\s_{n-1}u\s_{n+1} -
u\s_{n+1}w\s_{n+2} + \frac{1}{8} u\s_{n-1}w\s_{n-2}
+\frac{1}{4} v\s_{n-1}w\s_{n+1} -\\
&\quad\quad - w\s_{n+1}v\s_{n+2} + \frac{1}{8} w\s_{n-1}v\s_{n-2}\Big]e_n\\
&=i\sum_{n=1}^{\infty}k_n\Big[\frac{1}{4} u\s_{n-1}u\s_{n+1} - \frac{1}{4} v\s_{n-1}u\s_{n+1} - u\s_{n+1}u\s_{n+2} + u\s_{n+1}v\s_{n+2}+\\
&\quad\quad + \frac{1}{8} u\s_{n-1}u\s_{n-2} - \frac{1}{8} u\s_{n-1}v\s_{n-2} + \frac{1}{4} v\s_{n-1}u\s_{n+1} - \frac{1}{4} v\s_{n-1}v\s_{n+1}-\\
&\quad\quad - u\s_{n+1}v\s_{n+2} + v\s_{n+1}v\s_{n+2} + \frac{1}{8} u\s_{n-1}v\s_{n-2} - \frac{1}{8} v\s_{n-1}v\s_{n-2}\Big]e_n\\
&=i\sum_{n=1}^{\infty}k_n\Big[\frac{1}{4} u\s_{n-1}u\s_{n+1} - u\s_{n+1}u\s_{n+2} + \frac{1}{8} u\s_{n-1}u\s_{n-2}\Big]e_n\\
&\quad\quad -i\sum_{n=1}^{\infty}k_n\Big[\frac{1}{4} v\s_{n-1}v\s_{n+1} - v\s_{n+1}v\s_{n+2} + \frac{1}{8} v\s_{n-1}v\s_{n-2}\Big]e_n\\
&=B(u, u)-B(v, v).
\end{align*}
\end{proof}

With above functional setting and following the classical treatment of the Navier-Stokes equation, and in order to simplify the notation one can write the stochastic GOY model of turbulence \eqref{goy} in a Hilbert space $H$ in the following way,
\begin{align}\label{goy1}
&\d u + \big[\nu Au + B(u, u)\big] \d t = f(t) \d t + \sigma(t, u) \d W(t)\\
& u(0) = u_0,
\end{align}
where $u=(u_1, u_2,\ldots)\in H$, the operators $A$ and $B$ are
defined through \eqref{A} and \eqref{B} respectively,  $f=(f_1, f_2,
\ldots), \sigma(t, u)=(\sigma_1(t, u_1), \sigma_2(t, u_2), \ldots),$
and $W=(w_1, w_2, \ldots)$. Here $(W(t)_{t\geq 0})$ is a $H$-valued
Wiener process with trace class covariance. The noise coefficient
$\sigma : [0, T] \times V \to L_Q(H_0; H)$ is such that it satisfies
the following hypotheses:
\begin{enumerate}
\item[A.1.]  The function $\sigma \in C([0, T] \times V; L_Q(H_0; H))$
\item[A.2.]  For all $t \in (0, T)$, there exists a positive constant $K$ such that $|\sigma(t, u)|_{L_Q}^2 \le K(1 +||u||^2)$.
\item[A.3.]  For all $t \in (0, T)$,  there exists a positive constant $L$ such that for all $u, v \in V$,
$|\sigma(t, u) - \sigma(t, v)|^2_{L_Q} \le L||u - v||^2$ .
\end{enumerate}
\begin{remark}
The above hypotheses can be verified from the assumptions $(a.1. -
a.2.)$ on the noise coefficients. Notice that, $Q: H \to H$ is a
trace class covariance (nuclear) operator and hence compact. So $H_0
= Q^{1/2}H$ is a separable Hilbert space and the imbedding of $H_0$
in $H$ is Hilbert-Schmidt. Let $\{e_n\}_{n=1}^{\infty}$ be the
eigenfunctions of $Q$ (may not be complete). Then $Qe_n = \lambda_n e_n$, where each
$\lambda_n$ is positive real and $\sum_n \lambda_n < \infty$. Note,
\begin{align*}
|\sigma(t, u)|_{L_Q}^2 &=\sum_{m, n=1}^{\infty}|(\sigma h_m, e_n)|^2= \sum_{m, n=1}^{\infty} \lambda_m |(\sigma e_m, e_n)|^2\\
 &= (\sigma Q^{1/2}, \sigma Q^{1/2})=\Tr(\sigma Q \sigma),
\end{align*}
where  $\{h_m\}$, with $h_m = \sqrt{\lambda_m}e_m, m=1, 2, \ldots$ are orthonormal basis in $H_0$.

Then, using assumption $(a.1.)$ and letting $\lambda = \sup_{1\leq
n< \infty}\lambda_n < \infty$, one can have
\begin{align*}
|\sigma(t, u)|_{L_Q}^2&= \sum_{n=1}^{\infty} \lambda_n |\sigma_n(t, u_n)|^2
\leq (\sup_{1\leq n< \infty}\lambda_n) \sum_{n=1}^{\infty}|\sigma_n(t, u_n)|^2\nonumber\\
&\leq K_1 \lambda\sum_{n=1}^{\infty}k_n^2 |u_n|^2 \leq K(1+
\|u\|^2_V),
\end{align*}
which shows that hypothesis $(A.2)$ holds.

Similarly,
\begin{align*}
|\sigma(t, u) - \sigma(t, v)|^2_{L_Q}&\leq \lambda \sum_{n=1}^{\infty}|\sigma_n(t, u_n)-\sigma_n(t, v_n)|^2\nonumber\\
&\leq\lambda K_2 \sum_{n=1}^{\infty}k_n^2 |u_n - v_n|^2 = L \|u -
v\|^2_V.
\end{align*}
Thus the hypothesis $(A.3.)$ holds true.

Thus in the abstract setting of stochastic GOY model, the
assumptions $(a.1. - a.2.)$ are required on the noise coefficient to
impose the hypotheses $(A.1.-A.3.)$ in the Hilbert space valued
construction.
\end{remark}

In the following lemma we will show that sum of the linear and
nonlinear operator is locally monotone in the $l^4$-ball in $V$.
\begin{lemma}\label{Mon}
For a given $r > 0$, let us denote by $\mathbb{B}_r$ the closed
$l^4$-ball in $V$:
$$\mathbb{B}_r = \Big\{v\in V; \|v\|_{l^4} \leq r\big\}.$$
Define the nonlinear operator F on $V$ by $F(u):=-\nu Au - B(u, u)$. Then for any $0 < \varepsilon <\frac{\nu}{2L}$, where $L$ is the positive constant that appears in the condition (A.3), the pair $(F, \sqrt{\varepsilon}\sigma)$ is monotone in $\mathbb{B}_r$, i.e. for any $u\in V$ and $v\in \mathbb{B}_r$
\begin{equation}\label{monotone}
(F(u) - F(v), w)  - \frac{r^4}{\nu^3} |w|^2 + \varepsilon |\sigma(t, u) - \sigma(t, v)|^2_{L_Q} \leq 0,
\end{equation}
where $w = u - v$.
\end{lemma}
\begin{proof}
First note that, $$\nu(Aw, w) = \nu\|w\|^2.$$
Next using the Lemma \ref{Bprop2} and equation\eqref{1} from Lemma \ref{Bprop1}, we have
$$(B(u, u) - B(v, v), w) = (B(v, w) + B(w, v) + B(w, w), w) = (B(w, v), w).$$
Now using the definition of the operator $B$ and equation \eqref{L4} from Lemma \ref{La}, we get for $C >0$,
\begin{align*}
\big|(B(w, v), w)\big| &= \big|\sum_{n=1}^{\infty} ik_n\big[\frac{1}{4}v\s_{n-1}w\s_{n+1} w\s_{n} - \frac{1}{2}(w\s_{n+1} v\s_{n+2} + w\s_{n+2} v\s_{n+1})w\s_n + \\
&\quad + \frac{1}{8}w\s_{n-1} v\s_{n-2}w\s_n\big]\big|\\
&\leq C\|v\|_{l^4} \|w\|_{l^4} \|w\|\\
&\leq \|v\|_{l^4} |w|^{1/2} \|w\|^{3/2}\\
&\leq \frac{\nu}{2}\|w\|^2 + \frac{27}{32\nu^3} |w|^2 \|v\|_{l^4}^4.
\end{align*}
Since $v\in\mathbb{B}_r$, the above relation yields
$$- (B(w, v), w) \leq \frac{\nu}{2}\|w\|^2 + \frac{r^4}{\nu^3} |w|^2.$$
Hence by the definition of the operator $F$,
\begin{equation}\label{monotone2}
(F(u) - F(v), w) \leq -\frac{\nu}{2}\|w\|^2 + \frac{r^4}{\nu^3} |w|^2.
\end{equation}
Finally, using condition (A.3) and that $\varepsilon <\frac{\nu}{2L}$ we get the desired result.
\end{proof}

\subsection{Energy estimate and existence theory}
Let $H_n := \text{span}\ \{e_1, e_2, \cdots, e_n\}$ where $\{e_j\}$ is any fixed orthonormal basis in $H$ with each $e_j \in D(A)$. Let $P_n$ denote the orthogonal projection of $H$ to $H_n$.
Define $u^n = P_n u$, not to cause any confusion in notation with earlier $u_n$. Let $W_n = P_nW$. Let $\sigma_n = P_n\sigma $.
Define $\u$ as the solution of the following stochastic differential equation in the variational form such that for each $v \in H_n$,
\begin{equation}\label{variational}
\d (\u(t) , v) = (F(\u(t)), v)\d t + (f(t), v)\d t + \sqrt{\varepsilon}(\sigma_n(t, \u(t)) \d W_n(t), v),
\end{equation}
with $\u(0) = P_n u(0) $.

%Theorem %%%%%%%%%%%%%%%%%%%%%%%%%%%%%%%%%%%%%%%%%%%%%%%%%
\begin{theorem}\label{energy}
Under the above mathematical setting let $f$ be in $\mathrm{L}^2( [0, T], H)$ and let $E|u (0)|^2 < \infty$.
Let $\u$ denote the unique strong solution of the stochastic differential equation \eqref{variational} in $C([0, T], H_n)$. Then with $K$ as in condition (A.2), the following estimates hold:

For all $\varepsilon < \frac{\nu}{2K}$, and $0 \leq t \leq T$,
\begin{align}
E| \u(t)|^2 + \frac{\nu}{2}\int_0^t E\| \u(s)\|^2 \d s \leq E|u(0)|^2 + \frac{1}{\nu} \int_0^t \|f(s)\|^2_{V^{\prime}}\d s + \varepsilon KT,\label{energy1}
\end{align}
and for all $\varepsilon < \frac{\nu}{2K} \wedge \frac{1}{2K^2}$
\begin{align}
E\Big[\sup_{0\leq t\leq T} |\u(t)|^2 + \frac{\nu}{2}\int_0^T \|\u(t)\|^2 \d t\Big] \leq C\Big(E|u(0)|^2, \int_0^T \|f(t)\|^2_{V^{\prime}} \d t, \nu, T\Big)\label{energy2}.
\end{align}
Also, for any $\delta > 0$ and $\varepsilon < \frac{3\nu}{2K}$,
\begin{align}
E|\u(t)|^2 e^{-\delta t} + \frac{\nu}{2}\int_0^T E\|\u(t)\|^2 e^{-\delta t} \d t &\leq E|u(0)|^2 + \frac{1}{\delta}\int_0^T |f(t)|^2 e^{-\delta t}\d t +\nonumber\\
&\quad + \frac{\varepsilon K}{\delta}.\label{energy3}
\end{align}
Moreover, if we suppose that $f \in \mathrm{L}^4( [0, T], H)$ and $E|u (0)|^4 < \infty$, then for all $\varepsilon < \frac{\nu}{3K}$ and $0 \leq t \leq T$,
\begin{align}
&E\Big[\sup_{0 \leq t \leq T}|\u(t)|^4 e^{-\delta t} + C_{\nu}\int_0^T \|\u(t)\|^2 |\u(t)|^2 e^{-\delta t} \d t\Big]\nonumber \\
&\quad\leq E|u(0)|^4 + C_{\delta, T}\int_0^T \|f (t)\|^4_{V^{\prime}} e^{-\delta t} \d t + \frac{\varepsilon M}{\delta}.\label{energy4}
\end{align}
\end{theorem}

\begin{proof}
Replacing $v$ with $\u$ in \eqref{variational} and using the properties of the operators $A$ and $B$, we notice that,
\begin{align*}
&\d |\u(t)|^2 + 2\nu\|\u(t)\|^2 \d t \\
&\quad = 2(f(t),
\u(t))\d t + \varepsilon \Tr (\sigma_n(t,\u(t)) Q \sigma_n(t,\u(t)))\d t+\\
&\quad\quad + 2\sqrt{\varepsilon} (\sigma_n(t, \u(t)), \u(t)) \d W_n(t).
\end{align*}
Using the inequality $$ 2ab \leq \delta a^{2} + \frac{1}{\delta}
b^{2}$$ on $2 (f(t), \u(t))$ and using the condition (A.2), we obtain
\begin{align*}
\d |\u(t)|^2 + 2\nu\|\u(t)\|^2 \d t &\leq (\nu\|\u(t)\|^2 + \frac{1}{\nu}\|f(t)\|^2_{V^{\prime}}) \d t+ \varepsilon K(1 + \|\u(t)\|^2) \d t\\
&\quad + 2\sqrt{\varepsilon} (\sigma_n(t, \u(t)), \u(t)) \d W_n(t).
\end{align*}
Integrating in $0\leq t\leq T$ and taking the expectation and using
a stopping time argument, one can deduce
\begin{align*} E|\u(t)|^2 +
(\nu - \varepsilon K)\int_0^t\|\u(s)\|^2 \d s &\leq E|u(0)|^2 +
\frac{1}{\nu}\int_0^t\|f(s)\|^2_{V^{\prime}} \d s + \varepsilon K T.
\end{align*}
Now for all $\varepsilon < \frac{\nu}{2K}$, we get the desired result \eqref{energy1}.

\sni
To prove \eqref{energy2}, we proceed in the similar way as above, but we take supremum in time $0 \leq t\leq T$ before taking the expectation to get,
\begin{align}
&E\Big[\sup_{0\leq t\leq T}|\u(t)|^2 + \frac{\nu}{2}\int_0^T\|\u(t)\|^2 \d t\Big]\nonumber \\
&\quad\leq E|u(0)|^2 + \frac{1}{\nu}\int_0^T\|f(t)\|^2_{V^{\prime}} \d t + \varepsilon K T +\nonumber \\
&\quad\quad + 2\sqrt{\varepsilon} E\Big[\sup_{0\leq t\leq T}\Big|\int_0^t(\sigma_n(s, \u(s)), \u(s)) \d W_n(s)\Big|\Big].\label{3}
\end{align}
By means of Burkholder-Davis-Gundy inequality, condition (A.2) and Cauchy-Schwartz inequality,
\begin{align}
&2\sqrt{\varepsilon} E\Big[\sup_{0\leq t\leq T}\Big|\int_0^t(\sigma_n(s, \u(s)), \u(s)) \d W_n(s)\Big|\Big]\nonumber\\
&\quad\leq 2\sqrt{2\varepsilon} K E\Big[\Big(\int_0^{T} (1+\|\u(t)\|^2) |\u(t)|^2 \d t \Big)^{1/2}\Big]\nonumber\\
&\quad\leq 2\sqrt{2\varepsilon} K E\Big[\sup_{0\leq t\leq T} |\u(t)| \Big(\int_0^{T} (1+\|\u(t)\|^2) \d t\Big)^{1/2}\Big]\nonumber\\
&\quad\leq \sqrt{2\varepsilon} K \Big[E (\sup_{0\leq t\leq T} |\u(t)|^2) + E\int_0^{T} \|\u(t)\|^2 \d t + T\Big].\label{4}
\end{align}
Using \eqref{4} in \eqref{3}, one can get the desired energy estimate \eqref{energy2} if $\sqrt{2\varepsilon} K < 1$.

\sni
Next, we consider the function $e^{-\delta t} |\u(t)|^2$ for $\delta>0$ and apply the It\^{o} Lemma to get,
\begin{align}
&\d \Big[|\u(t)|^2 e^{-\delta t}\Big] + 2\nu\|\u(t)\|^2e^{-\delta t}\d t + \delta |\u(t)|^2 e^{-\delta t}\d t\nonumber \\
&\quad=\Big[ 2(f(t),
\u(t)) + \varepsilon \Tr (\sigma_n(t,\u(t)) Q \sigma_n(t,\u(t)))\Big]e^{-\delta t}\d t\nonumber\\
&\quad\quad + 2\sqrt{\varepsilon} (\sigma_n(t, \u(t)), \u(t))e^{-\delta t} \d W_n(t).\label{5}
\end{align}
Note that $$2(f(t),\u(t)) \leq \delta |\u(t)|^2 + \frac{1}{\delta} |f(t)|^2.$$
Hence upon writing \eqref{5} in the integral form, taking expectation and using condition (A.2), one can get
\begin{align*}
& E |\u(t)|^2 e^{-\delta t} + 2\nu E\int_0^T\|\u(t)\|^2e^{-\delta t}\d t\nonumber \\
&\quad\leq E|u(0)|^2 + \frac{1}{\delta}\int_0^T |f(t)|^2 e^{-\delta t} \d t + \varepsilon K E\int_0^T (1+\|\u(t)\|^2 e^{-\delta t} \d t,
\end{align*}
which yields the estimate \eqref{energy3} for all $\e < \frac{3\nu}{2K}$.

\sni
To prove \eqref{energy4}, we first apply It\^{o} Lemma on the function $|\u(t)|^4 e^{-\delta t}$ to get,
\begin{align}
&\d \Big[|\u(t)|^4 e^{-\delta t}\Big] + 4\nu\|\u(t)\|^2 |\u(t)|^2 e^{-\delta t}\d t + \delta |\u(t)|^4 e^{-\delta t}\d t\nonumber \\
&\quad=|\u(t)|^2\Big[4(f(t),
\u(t)) + 8\varepsilon \Tr (\sigma_n(t,\u(t)) Q \sigma_n(t,\u(t)))\Big]e^{-\delta t}\d t\nonumber\\
&\quad\quad + 4\sqrt{\varepsilon} (\sigma_n(t, \u(t)), \u(t)) |\u(t)|^2 e^{-\delta t} \d W_n(t).\label{6}
\end{align}
Using the fact that $$4(f(t), \u(t))|\u(t)|^2 \leq C_{\delta} |f(t)|^4 + \delta |\u(t)|^4,$$
and applying the condition (A.2) and integrating we have,
\begin{align*}
&|\u(t)|^4 e^{-\delta t} +(4\nu -8\e K)\int_0^t\|\u(s)\|^2 |\u(s)|^2 e^{-\delta s}\d s\nonumber \\
&\quad\leq E|u(0)|^4 + C_{\delta}\int_0^t |f(s)|^4 e^{-\delta s}\d s\nonumber\\
&\quad\quad + 4\sqrt{\e} \int_0^t(\sigma_n(s, \u(s)), \u(s)) |\u(s)|^2 e^{-\delta s} \d W_n(s).
\end{align*}
Finally taking supremum in $0\leq t\leq T$, then taking expectation on both sides and using the Burkholder-Davis-Gundy inequality on the stochastic integral term, we get the estimate \eqref{energy4}.
\end{proof}

\begin{definition}($Strong\ Solution$)
A strong solution $\ue$ is defined on a given probability
space $(\Omega, \mathcal{F}, \mathcal{F}_{t}, P)$ as a
$\mathrm{L}^2(\Omega;\mathrm{L}^{\infty
}(0,T; H)\cap
\mathrm{L}^2(0,T; V)\cap
C(0,T; H))$ valued function which
satisfies the stochastic GOY model
\begin{align}\label{7}
&\d \ue + \big[\nu A\ue + B(\ue, \ue)\big] \d t = f(t) \d t + \sqrt{\e}\sigma(t, \ue) \d W(t)\\
& \ue(0) = u_0,\nonumber
\end{align}
 in the
weak sense and also the energy inequalities in Theorem \ref{energy}.
\end{definition}
Monotonicty arguments were first used by Krylov and
Rozovskii\cite{KR} to prove the existence and uniqueness of the
strong solutions for a wide class of stochastic evolution equations
(under certain assumptions on the drift and diffusion coefficients),
which in fact is the refinement of the previous results by
Pardoux\cite{Pa1, Pa2} and also the generalization of the results by
Bensoussan and Temam\cite{Be}. Menaldi and Sritharan\cite{Ms}
further developed this theory for the case when the sum of the
linear and nonlinear operators are locally monotone.

%%%%Theorem %Existence &  Uniqueness of the strong solution %%%%%%%%%%%%%%%%%%%%%%%%%%%%%%%%%%%%%%%%%%%%%%%%%%%%%%%%%%%%%%%%%%%%%%%%%%%%%%%%%%

\begin{theorem}\label{existence}
Let the data $f$ and $u_0$ be such that
$$f\in \mathrm{L}^4(0, T; V^{\prime}),\ E|u_0|^4 < \infty.$$
We also assume that $0 < \e < \frac{\nu}{L}$ and the diffusion coefficient satisfies the conditions (A.1)-(A.3). Then almost surely there exists a unique adapted process $\ue(t, x, w)$ with the regularity
$$\ue \in \mathrm{L}^2(\Omega; C(0, T; H) \cap \mathrm{L}^2(0, T; V))$$
satisfying the stochastic GOY model \eqref{7} and the a priori bounds in Theorem \ref{energy}.
\end{theorem}

\begin{proof}
Using the a priori estimate in the Theorem \ref{energy}, it follows from the Banach-Alaoglu theorem
that along a subsequence, the Galerkin approximations
$\{\u\}$ have the following limits:
\begin{align}
&\u\longrightarrow \ue\quad  \text {weak star in}\ \mathrm{L}^4(\Omega ;
  \mathrm{L}^{\infty}(0,T; H)) \cap\mathrm{L}^{2}
  (\Omega;\mathrm{L}^{2}(0,T; V)),\nonumber\\
& F(\u)\longrightarrow F^{\e}_0\quad \text{weakly in}\
\mathrm{L}^{2}(\Omega;\mathrm{L}^{2}(0,T; V^{\prime})),\nonumber\\
& \sigma_n(\cdot, \u) \longrightarrow S^{\e}\quad \text{weakly in}\ \mathrm{L}^{2}(\Omega;\mathrm{L}^{2}(0,T; \mathrm{L}_Q)).
\end{align}
The assertion of the second statement holds since $F(\u)$ is bounded in \\ $\mathrm{L}^{2}(\Omega;\mathrm{L}^{2}(0,T; V^{\prime}))$. Likewise since diffusion coefficient has the linear growth property and $\u$ is bounded in $\mathrm{L}^2(0, T; V)$ uniformly in $n$, the last statement holds. Then $\ue$ has the It\^{o} differential
\begin{align*}
\d \ue(t) = F^{\e}_0(t)\d t + \sqrt{\e} S^{\e}(t)\d W(t)\quad \text{weakly in}\
\mathrm{L}^{2}(\Omega;\mathrm{L}^{2}(0,T; V^{\prime})).
\end{align*}
Let us set,
\begin{equation}\label{9}
r(t):= \frac{2}{\nu^3}\int_0^t \|\ue(s)\|_{\mathrm{L}^4}^4 \d s.
\end{equation}
Here we suppress the dependence of $\varepsilon$ in the notation of $r$ to make it easier to read. Then applying the It\^{o} Lemma to the function $2e^{-r(t)} |\u(t)|^2$, one obtains
\begin{align*}
\d \big[e^{-r(t)} |\u(t)|^2\big] &= e^{-r(t)}\big(2F(\u(t)) - \dot{r}(t) \u(t), \u(t)\big)\d t +\\
&\quad + \e e^{-r(t)} |\sigma_n(t, \u(t))|^2_{\mathrm{L}_Q}\d t +\\
& \quad + 2\sqrt{\e} e^{-r(t)}\big(\sigma_n(t, \u(t)), \u(t)\big)\d W(t).
\end{align*}
Integrating between $0 \leq t\leq T$ and taking expectation,
\begin{align*}
&E\big[e^{-r(T)} |\u(T)|^2 - |\u(0)|^2\big]\\
&\quad= E\big[\int_0^T e^{-r(t)}\big(2F(\u(t)) - \dot{r}(t) \u(t), \u(t)\big)\d t\big]+\\
&\quad\quad +\e E\int_0^T e^{-r(t)} |\sigma_n(t, \u(t))|^2_{\mathrm{L}_Q}\d t+\\
&\quad\quad + 2\sqrt{\e} E\int_0^T e^{-r(t)}\big(\sigma_n(t, \u(t)), \u(t)\big)\d W(t).
\end{align*}
The last term on the right hand side vanishes since the integral inside the expectation is a martingale. Then by the lower semi-continuity property of the weak convergence,
\begin{align}
&\liminf_n E\big[\int_0^T e^{-r(t)}\big(2F(\u(t)) - \dot{r}(t) \u(t), \u(t)\big)\d t +\nonumber \\
&\quad +\e \int_0^T e^{-r(t)} |\sigma_n(t, \u(t))|^2_{\mathrm{L}_Q}\d t\big]\nonumber\\
&\quad\quad = \liminf_n E\big[e^{-r(T)} |\u(T)|^2 - |\u(0)|^2\big]\nonumber\\
&\quad\quad \geq E\big[e^{-r(T)} |\ue(T)|^2 - |\ue(0)|^2\big]\nonumber\\
&\quad\quad = E\big[\int_0^T e^{-r(t)}\big(2F_0^{\e}(t) - \dot{r}(t) \ue(t), \ue(t)\big)\d t +\e \int_0^T e^{-r(t)} |S^{\e}|^2_{\mathrm{L}_Q}\d t\big].\label{lsc}
\end{align}
Now by monotonicity property from Lemma \ref{Mon},
\begin{align*}
& 2E\big[\int_0^T e^{-r(t)} \big(F(\u(t)) - F(\ve(t)), \u(t) - \ve(t)\big)\d t\big] - \\
&\quad - E\big[\int_0^T e^{-r(t)} \dot{r}(t)|\u(t) - \ve(t)|^2 \d t\big] + \\
&\quad + \e E\big[\int_0^T e^{-r(t)} |\sigma_n(t, \u(t)) - \sigma_n(t, \ve(t))|^2_{\mathrm{L}_Q}\d t\big]\\
&\quad\quad \leq 0.
\end{align*}
Rearranging the terms,
\begin{align*}
& E\big[\int_0^T e^{-r(t)}\big(2F(\u(t)) - \dot{r}(t) \u(t), \u(t)\big)\d t +\\
&\quad +\e \int_0^T e^{-r(t)} |\sigma_n(t, \u(t))|^2_{\mathrm{L}_Q}\d t\big]\\
&\quad\quad\leq E\big[\int_0^T e^{-r(t)}\big(2F(\u(t))-\dot{r}(t)(2\u(t) - \ve(t)), \ve(t)\big) \d t\big] +\\
&\quad\quad\quad + E\big[\int_0^T e^{-r(t)}\big(2F(\ve(t)), \u(t)-\ve(t)\big) \d t\big] +\\
&\quad\quad\quad + \e E\big[\int_0^T e^{-r(t)}\big(2\sigma_n(t, \u(t))-\sigma_n(t, \ve(t)), \sigma_n(t, \ve(t))\big)_{\mathrm{L}_Q}\d t\big].
\end{align*}
Taking limit in $n$, using the result from \eqref{lsc} and rearranging, we obtain
\begin{align*}
&E\big[\int_0^T e^{-r(t)}\big(2F_0^{\e}(t)-2F(\ve(t)), \ue(t)-\ve(t)\big)\d t\big] +\\
&\quad + E\big[\int_0^T e^{-r(t)}\dot{r}(t) |\ue(t) - \ve(t)|^2 \d t\big]+\\
&\quad + \e E\big[\int_0^T e^{-r(t)} \|S(t)-\sigma(t, \ve(t))\|^2_{\mathrm{L}_Q}\d t\big]\\
&\quad\quad \leq 0.
\end{align*}
Notice that for $\ve=\ue$, $S(t)=\sigma(t, \ue(t))$. Take $\ve = \ue - \lambda\we$ with $\lambda >0$ and $\we$ is an adapted process in $\mathrm{L}^2(\Omega; C(0, T; H) \cap \mathrm{L}^2(0, T; V))$ Then,
\begin{align*}
&\lambda E\big[\int_0^T e^{-r(t)}\big(2F_0^{\e}(t)-2F(\ue - \lambda\we)(t), \we(t)\big)\d t + \lambda\int_0^T e^{-r(t)}\dot{r}(t)|\we(t)|^2\d t\big]\\ &\quad\leq 0.
\end{align*}
Dividing by $\lambda$ on both sides of the inequality above and letting $\lambda$ go to $0$, one obtains
\begin{align*}
E\big[\int_0^T e^{-r(t)}\big(F_0^{\e}(t)-F(\ue(t)), \we(t)\big)\d t\big] \leq 0.
\end{align*}
Since $\we$ is arbitrary, we conclude that $F_0^{\e}(t)=F(\ue(t))$. Thus the existence of the strong solution of the stochastic GOY model \eqref{7} has been proved.

If $\ve\in \mathrm{L}^2(\Omega; C(0, T; H) \cap \mathrm{L}^2(0, T; V))$ be another solution of the equation \eqref{7} then $\we=\ue-\ve$ solves the stochastic differential equation in $\mathrm{L}^2(\Omega;\mathrm{L}^2(0, T; V^{\prime}))$,
\begin{equation}\label{8}
\d \we(t) = (F(\ue(t)) - F(\ve(t)))\d t + \sqrt{\e}(\sigma(t, \ue(t)) - \sigma(t, \ve(t))) \d W(t).
\end{equation}
We denote $\sigma_d = \sigma(t, \ue(t)) - \sigma(t, \ve(t))$.

We now apply It\^{o} Lemma to the function $2e^{-r(t)} |\we(t)|^2$
and using the local monotonicity of the sum of the linear and
nonlinear operators $A$ and $B$, e.g. equation \eqref{monotone2}, we
get
\begin{align}
e^{-r(t)} \d |\we(t)|^2 + \nu e^{-r(t)} \|\we(t)\|^2  \d t &\leq \frac{2r^4}{\nu^3}e^{-r(t)} |\we(t)|^2 \d t + \e e^{-r(t)} \Tr(\sigma_d Q \sigma_d) \d t + \nonumber\\
&\quad + 2\sqrt{\e} e^{-r(t)} (\sigma_d, \we(t)) \d W(t).
\end{align}
Using condition (A.3),
\begin{align}
\d (e^{-r(t)}|\we(t)|^2)+ \nu e^{-r(t)} \|\we(t)\|^2  \d t &\leq \e L e^{-r(t)} \|\we(t)\|^2 \d t + \nonumber\\
&\quad + 2\sqrt{\e} e^{-r(t)} (\sigma_d, \we(t)) \d W(t).
\end{align}
Finally integrating in $0\leq t\leq T$, taking expectation on both sides and noting $\e < \frac{\nu}{L}$ and the fact that
$$\int_0^T e^{-r(t)} (\sigma_d, \we(t)) \d W(t)$$ is a martingale for $T < \infty$, we obtain P-a.s.
\begin{align*}
E\big[e^{-r(t)}|\we(t)|^2\big] \leq E|w(0)|^2,
\end{align*}
which assures the uniqueness of the strong solution.
\end{proof}

\section{Large Deviation Principle Continued}
\setcounter{equation}{0}
Let us recall the stochastic GOY model in consideration,
\begin{align}\label{10}
&\d \ue + \big[\nu A\ue + B(\ue, \ue)\big] \d t = f(t) \d t + \sqrt{\e}\sigma(t, \ue) \d W(t)\\
& \ue(0) = \xi,\nonumber
\end{align}
The aim of this section is to prove the LDP for $\{\ue : \e>0\}$ in
$X=C([0, T]; H) \cap \mathrm{L}^2(0, T; V)$ by verifying the
Assumption \ref{assum} and then applying the Theorem \ref{main},
which has already been mentioned in Remark \ref{rem1}.

The LDP for $\{\ue : \e>0\}$in $X$ have been proved here systematically in four steps. In the first and second Theorems we show the well
posedness of certain controlled stochastic and controlled deterministic equations in $X$. These results help to prove the last two main Theorems
on the compactness of the level sets and weak convergence of the stochastic control equation stated in Assumption \ref{assum}.

\begin{theorem}\label{scontrol}
Let the family $\mathcal{G}^{\e}$ be defined as in Section 2.  For
any $v \in \mathcal{A}_M$, where $0 < M < \infty$, let
$\mathcal{G}^{\e}\big(W(\cdot) + \frac{1}{\sqrt{\e}}\
\int_0^{\cdot}\ v(s)\d s\big)$ be denoted by $\uv$ where $\uv(0) =
\xi\in H$. Then
\begin{align}
\d \uv(t) + [\nu A\uv(t) + B(\uv(t), \uv(t))]\d t &= [f(t) + \st(t,
\uv(t))v(t)]\d t + \nonumber\\
&\quad + \sqrt{\e}\sigma(t, \uv(t)) \d W(t),\label{11}
\end{align}
has a unique strong solution in $\mathrm{L}^2(\Omega; X)$, where
$X=C(0, T; H) \cap \mathrm{L}^2(0, T; V)$, $f \in \mathrm{L}^4(0, T;
V^{\prime})$ and $\sigma, \st$ both satisfy the hypotheses
A.1.--A.3. in Section 3.
\end{theorem}

\begin{proof}
We first prove that if $\uv (t)$ is a strong solution of the
stochastic controlled equation \eqref{11}, the following energy
estimate holds:
\begin{align}\label{12}
E\Big(\sup_{0\leq t\leq T} |\uv (t)|^2 + \int_0^T \|\uv (t)\|^2 \d
t\Big) \leq C,
\end{align}
where $$C = C\big(|\xi|^2, \int_0^T \|f\|^2_{V^{\prime}} \d t, \nu,
K, T, M\big)$$ is a positive constant.

To prove the above estimate, we start by taking the inner product of
the equation \eqref{11} with $\uv(t)$ and integrating in $0\leq
t\leq T$,
\begin{align}\label{13}
&|\uv(t)|^2 + \nu\int_0^t \|\uv(s)\|^2 \d s\nonumber \\
&\quad\leq |\xi|^2 + \frac{1}{\nu}\int_0^t \|f(s)\|^2_{V^{\prime}}
\d
s + 2\int_0^t \big(\st(s, \uv(s))v(s), \uv(s)\big)\d s + \nonumber\\
&\quad\quad +\e\int_0^t \Tr(\sigma(s, \uv(s)) Q \sigma(s, \uv(s)))
\d s + 2\sqrt{\e}\int_0^t \big(\sigma(s, \uv(s)), \uv(s)\big) \d
W(s).
\end{align}
Notice that,
\begin{align}
&2\int_0^t \big(\st(s, \uv(s))v(s), \uv(s)\big)\d s\nonumber\\
&\quad \leq 2\int_0^t |\st(s, \uv(s))|_{\mathrm{L}_Q} |v(s)|_0
|\uv(s)| \d s\nonumber\\
&\quad\leq \frac{\nu}{4K}\int_0^t |\st(s, \uv(s))|_{\mathrm{L}_Q}^2
\d s + \frac{4K}{\nu}\int_0^t |v(s)|_0^2 |\uv(s)|^2 \d s\nonumber\\
&\quad\leq \frac{\nu}{4}\int_0^t (1+\|\uv(s)\|^2) \d s +
\frac{4K}{\nu}\int_0^t |v(s)|_0^2 |\uv(s)|^2 \d
s\nonumber\\
&\quad\leq \frac{\nu}{4}\int_0^t (1+\|\uv(s)\|^2) \d s +
\frac{4KM}{\nu}\sup_{0\leq t\leq T} |\uv(t)|^2,\label{14}
\end{align}
and for $\e < \frac{\nu}{4K}$,
\begin{align}
\e\int_0^t \Tr(\sigma(s, \uv(s)) Q \sigma(s, \uv(s))) \d s \leq
\frac{\nu}{4}\int_0^t (1+\|\uv(s)\|^2) \d s.\label{15}
\end{align}
After rearrangement of the equation \eqref{13}, we take supremum in
time $0\leq t\leq T$ and then expectation to get,
\begin{align}
&E\Big[\sup_{0\leq t\leq T}|\uv(t)|^2 + \frac{\nu}{2}\int_0^T
\|\uv(t)\|^2 \d t\Big]\nonumber\\
&\quad\leq |\xi|^2 + \frac{\nu}{2}T + \frac{1}{\nu}\int_0^T
\|f(t)\|^2_{V^{\prime}} \d t + \frac{4KM}{\nu}E\Big[\sup_{0\leq
t\leq
T} |\uv(t)|^2\Big]+\nonumber\\
&\quad\quad + 2\sqrt{\e}E\Big[\sup_{0\leq t\leq T}\Big|\int_0^t
\big(\sigma(s, \uv(s)), \uv(s)\big) \d W(s)\Big|\Big].\label{16}
\end{align}
The last term of the above equation can be estimated in a similar
manner as in \eqref{4},
\begin{align}
&2\sqrt{\e}E\Big[\sup_{0\leq t\leq T}\Big|\int_0^t \big(\sigma(s,
\uv(s)), \uv(s)\big) \d W(s)\Big|\Big]\nonumber\\
&\quad\leq \sqrt{2\e}K\Big[E(\sup_{0\leq t\leq T} |\uv(t)|^2) +
E\int_0^T \|\uv(t)\|^2\d t + T\Big].\label{17}
\end{align}
Replacing \eqref{17} in \eqref{16}, and considering $$\e < \min
\Big\{\frac{\nu}{4K}, \frac{\nu^2}{8K^2},
\frac{(1-4KM/\nu)^2}{2K^2}\Big\},$$ we get the energy estimate
\eqref{12}.

The proof of the existence and uniqueness of the strong solution of
the stochastic controlled equation \eqref{11} follow from the
Theorem \ref{existence}, only a few modifications are needed due to
the presence of the control term.
\end{proof}

\begin{theorem}\label{dcontrol}
Let $v \in \mathrm{L}^2(0, T; H_0)$, $f \in \mathrm{L}^4(0, T;
V^{\prime})$ and $\sigma$ satisfy the hypotheses A.1.--A.3. in
Section 3. Then $u_v \in X=C(0, T; H) \cap \mathrm{L}^2(0, T; V)$ is
the unique strong solution of the equation
\begin{align}\label{18}
\d u_v(t) + [\nu Au_v(t)+ B(u_v(t), u_v(t))]\d t = f(t)\d t + \sigma(t,
u_v(t)) v(t) \d t,
\end{align}
where $u_v(0)=\xi\in H$.
\end{theorem}

\begin{proof}
This result can be considered as a particular case of the previous
Theorem \ref{scontrol}, where the diffusion coefficient is absent.
\end{proof}

Next we state a important lemma from Budhiraja and Dupuis
\cite{BD1}.

\begin{lemma}\label{BDL1}
Let $\{v_n\}$ be a sequence of elements from $\mathcal{A}_M$ for
some finite $M > 0$. Let $v_n$ converges in distribution to $v$ with
respect to the weak topology on $\mathrm{L}^2(0, T; H_0)$. Then
$\int_0^{\cdot}\ v_n(s)\d s$ converges in distribution as $C(0, T;
H)$- valued processes to $\int_0^{\cdot}\ v(s)\d s$ as $n \to
\infty$.
\end{lemma}

Now we are ready to check the Assumptions \ref{assum}.
\begin{theorem}[Compactness]\label{compact}
Let $M < \infty$ be a fixed positive number. Let
$$ K_M :=\Big\{u_v \in C(0, T; H) \cap \mathrm{L}^2(0, T; V); v\in
S_M\Big\},$$ where $u_v$ is the unique solution in $X=C(0, T; H)
\cap \mathrm{L}^2(0, T; V)$ of the deterministic controlled equation
\eqref{18}, with $u_v(0) = \xi\in H$. Then $K_M$ is compact in $X$.
\end{theorem}

\begin{proof}
Let us consider a sequence $\{\n\}$ in $K_M$, where $\n$ corresponds
to the solution of \eqref{18} with control $v_n \in S_M$ in place of
$v$, i.e.
\begin{align}\label{19}
\d \n(t) + [\nu A\n(t) + B(\n(t), \n(t))]\d t=f(t)\d t + \sigma(t,
\n(t))v_n(t)\d t,
\end{align}
with $\n(0)=\xi\in H$. Then by weak compactness of $S_M$, there
exists a subsequence of $\{v_n\}$, still denoted by $\{v_n\}$, which
converges weakly to $v\in S_M$ in $\mathrm{L}^2(0, T; H_0)$.

We need to prove $\n \to u_v$ in $X$ as $n\to\infty$, or in other
words,
\begin{align}\label{20}
\sup_{0\leq t\leq T} |\n(t) - u_v(t)|^2 + \int_0^T \|\n(t) -
u_v(t)\|^2 \d t \to 0,
\end{align}
as $n\to \infty$.

According to the Theorem \ref{dcontrol}, $u_v$ is unique strong
solution in $X$ of the deterministic controlled equation \eqref{18}.
Hence it is obvious to note that, $u_v$ satisfies the following
a-priori estimate
\begin{align}\label{21}
\sup_{0\leq t\leq T} |u_v (t)|^2 + \int_0^T \|u_v (t)\|^2 \d t \leq
C,
\end{align}
where $$C = C\big(|\xi|^2, \int_0^T \|f\|^2_{V^{\prime}} \d t, \nu,
K, T, M\big)$$ is a positive constant.

For the proof, we refer the Theorem \ref{scontrol}, where the
stochastic version of the above a priori estimate has been worked
out.

Let $\m = \n - u_v$. Then $\m$ satisfies the following differential
equation
\begin{align}
&\d \m(t) +[\nu A\m(t)+B(\n(t), \n(t)) - B(u_v(t), u_v(t))]\d
t\nonumber\\
&\quad = [\sigma(t, \n(t)) v_n(t) - \sigma(t, u_v(t)) v(t)]\d t,
\end{align}
which yields
\begin{align}
&|\m(t)|^2 + 2\nu\int_0^t\|\m(s)\|^2 \d s + \nonumber\\
&\quad + 2\int_0^t \Big( B(\n(s),
\n(s)) - B(u_v(s), u_v(s)), \m(s)\Big)\d s\nonumber\\
&\quad\quad = 2\int_0^t \Big(\sigma(s, \n(s)) v_n(s) - \sigma(s,
u_v(s)) v(s), \m(s)\Big)\d s.\label{22}
\end{align}
First note that, from Lemma \ref{Bprop2}, $$B(\n, \n) - B(u_v,
u_v)=B(u_v, \m) + B(\m, u_v) + B(\m, \m).$$ Using the above
expression and the properties $(ii)$ and $(iv)$ of the bilinear
operator $B$ given in Lemma \ref{Bprop1}, one can find
\begin{align}
& 2\Big| \Big( B(\n(s), \n(s)) - B(u_v(s), u_v(s)), \m(s)\Big)\Big|\nonumber\\
&\quad =2 \Big|\Big( B(\m(s), u_v(s)), \m(s)\Big)\Big|\nonumber\\
&\quad \leq 2 \|B(\m(s), u_v(s))\|_{V^{\prime}} \|\m(s)\| \nonumber\\
&\quad \leq 2 |\m(s)|\ |u_v(s)|\ \|\m(s)\| \nonumber\\
&\quad \leq \frac{\nu}{2}\|\m(s)\|^2  +
\frac{2}{\nu} |\m(s)|^2 |u_v(s)|^2.\label{23}
\end{align}
Also notice that,
\begin{align}
&\Big|\int_0^t \Big(\sigma(s, \n(s)) v_n(s) - \sigma(s, u_v(s)) v(s), \m(s)\Big)\d s \Big|\nonumber\\
& \leq  \int_0^t  \Big|\Big(\big(\sigma(s, \n(s)) - \sigma(s,
u_v(s))\big)v_n(s),
\m(s)\Big)\Big| \d s +\nonumber\\
&\quad + \Big|\int_0^t \Big(\sigma(s, u_v(s)) (v_n(s) - v(s)),
\m(s)\Big)\d s\Big|\nonumber\\
 & \leq \sqrt{L} \int_0^t  \|\m(s)\| \ |\m(s)| \
|v_n(s)|_0 \d s +\nonumber\\
&\quad + \Big|\int_0^t \Big(\sigma(s, u_v(s)) (v_n(s) - v(s)),
\m(s)\Big)\d s\Big| \nonumber\\
&\leq \frac{\nu}{4}\int_0^t\|\m(s)\|^2 \d s+ \frac{L}{\nu}\int_0^t
|\m(s)|^2 \ |v_n(s)|_0^2 \d s+\nonumber\\
&\quad +  \sup_{0 \le t \le T}\Big|\int_0^t \Big(\sigma(s, u_v(s))
(v_n(s) - v(s)), \m(s)\Big)\d s\Big|.\label{24}
\end{align}
By the boundedness of $\{|\m(s)|^2\}$ in $C(0, T; H)$, and using the
Lemma \ref{BDL1}, the second integral on the right side of
\eqref{24} goes to $0$ as $n \to \infty$.  Therefore, given any
$\epsilon > 0$, there exists an integer $N$ large so that  for all
$n \ge N$,
\begin{equation} \sup_{0 \le t \le T}\Big|\int_0^t \Big(\sigma(s, u_v(s)) (v_n(s) - v(s)),
\m(s)\Big)ds\Big| < \epsilon/2.\label{244}
\end{equation}
Consider, $$C_{L, \nu} = \max\Big\{\frac{2}{\nu},
\frac{2L}{\nu}\Big\}.$$ Applying \eqref{23},  \eqref{24} and
\eqref{244} in \eqref{22}, one obtains for $n \ge N$,
\begin{equation}\label{24a}
|\m(t)|^2 + \nu\int_0^t\|\m(s)\|^2 \d s \leq C_{L, \nu}\int_0^t
|\m(s)|^2\big(|u_v(s)|^2 + |v_n(s)|^2_0 \big) \d s + \epsilon
\end{equation}
Hence by Gronwall's inequality,
\begin{equation}
\sup_{0\leq t\leq T} |\m(t)|^2 + \nu\int_0^T\|\m(t)\|^2 \d t \leq
\epsilon  e^{C_{L, \nu}\int_0^T \big(|u_v(t)|^2 + |v_n(t)|^2_0 +
1\big)\d t}.\label{25}
\end{equation}
The arbitrariness of $\epsilon$ finishes the proof.
\end{proof}

\begin{theorem}[Weak convergence]
For any $v\in S_M$, $0<M<\infty$, let\\
$\mathcal{G}^0(\int_0^{\cdot}v(s) \d s)$ be denoted by $u_v$, where
$u_v$ is a unique strong solution in $X=C(0, T; H) \cap
\mathrm{L}^2(0, T; V)$ of the deterministic control equation
\eqref{18}. Let the family $\mathcal{G}^{\e}$ be defined as in
Section 2. Let $\{v^{\e} : \e
>0\} \subset \mathcal{A}_M$ converges in distribution to $v$ with respect to
the weak topology on $\mathrm{L}^2(0, T; H_0)$. Then $
\mathcal{G}^{\e}(W(\cdot) +
\frac{1}{\sqrt{\e}}\int_0^{\cdot}v^{\e}(s) \d s)$ converges in
distribution to $\mathcal{G}^0(\int_0^{\cdot}v(s) \d s)$ in X.
\end{theorem}

\begin{proof}
Let $\uve$ be the unique strong solution in $\mathrm{L}^2(\Omega;
X)$ of the equation
\begin{align}
&\d \uve(t) + [A\uve(t) + B(\uve(t), \uve(t))] \d t\nonumber\\
&\quad = [f(t)+\sigma(t, \uve(t))v^{\e}(t)]\d t + \sqrt{\e}\sigma(t,
\uve(t))\d W(t),\label{26}
\end{align}
with $\uve(0) = \xi\in H$.

Hence the following a priori estimate holds,
\begin{align}
E\Big[\sup_{0\leq t\leq T} |\uve(t)|^2 + \int_0^T \|\uve(t)\| \d
t\Big] \leq C\Big( |\xi|^2, \int_0^T \|f(t)\|_{V^{\prime}}^2 \d t,
K, T, M\Big).
\end{align}

Then there exist a Borel measurable function $\mathcal{G}^{\e}: C(0,
T; H) \to X$ satisfying the equality $\mathcal{G}^{\e}(W(\cdot) +
\frac{1}{\sqrt{\e}}\int_0^{\cdot}v^{\e}(s) \d s) = \uve.$ Since
$S_M$ is Polish, the Skorokhod representation theorem can be
introduced to construct processes $(\tilde{v}^{\e}, \tilde{v},
\tilde{W}^{\e})$ such that the  distribution of $(\tilde{v}^{\e},
\tilde{v}, \tilde{W}^{\e})$ is same as that of $(v^{\e}, v, W)$, and
$\tilde{v}^{\e} \to \tilde{v}$ a.s. in the weak topology of $S_M$.

Let $\w(t) = \uve(t)- u_v(t)$. We need to prove,
\begin{align}\label{27}
\sup_{0\leq t\leq T} |\w(t)|^2 + \int_0^T \|\w(t)\|^2 \d t \to 0
\end{align}
in probability as $\e \to 0$.

Notice that, applying similar estimate as in previous Theorem,
equation \eqref{24a}, one can get
\begin{align}
&|\w(t)|^2 + \nu\int_0^t \|\w(s)\|^2 \d s \nonumber\\
&\quad\leq 3 C_{L, \nu} \int_0^t |\w(s)|^2\big(|u_v(s)|^2 +
|v^{\e}(s)|_0^2 +1\big)\d s +\nonumber\\
&\quad\quad +\int_0^t |\sigma(s, u_v(s)) (v^{\e}(s) - v(s))|^2 \d s
+ \e K\int_0^t (1 + \|\uve(s)\|^2) \d s+\nonumber\\
&\quad\quad +2\sqrt{\e}\int_0^t \big(\sigma(s, \uve(s)),
\w(s)\big)\d W(s).\label{28}
\end{align}
We take supremum in $0\leq t\leq T$, then expectation on the above
inequality, and use similar estimate (with the help of
Burkholder-Davis-Gundy inequality) as in \eqref{4} on the last term
of the right hand side to get,
\begin{align}
&E\Big[\sup_{0\leq t\leq
T}|\w(t)|^2 + \nu\int_0^T \|\w(t)\|^2 \d t\Big] \nonumber\\
&\quad\leq 3C_{L, \nu} E\Big[\int_0^T \sup_{0\leq t\leq
T}|\w(t)|^2\big(|u_v(t)|^2 + |v^{\e}(t)|_0^2 +1\big)\d t\Big] +\nonumber\\
&\quad\quad +\int_0^T |\sigma(t, u_v(t)) (v^{\e}(t) - v(t))|^2 \d t
+ \e K(C+T)+\nonumber\\
&\quad\quad + \sqrt{2\e}K\Big(C+T+ E\Big[\sup_{0\leq t\leq
T}|\w(t)|^2\Big]\Big).
\end{align}
Assume that $\e <\frac{1}{2K^2}$. Then the Gronwall inequality
yields
\begin{align}
&E\Big[\sup_{0\leq t\leq
T}|\w(t)|^2 + \nu\int_0^T \|\w(t)\|^2 \d t\Big] \nonumber\\
&\quad\leq \Big((\e + \sqrt{2\e})K(C+T) + \int_0^T |\sigma(t,
u_v(t)) (v^{\e}(t) - v(t))|^2 \d t\Big)\times \nonumber\\
&\quad\quad\times e^{3C_{L, \nu}\int_0^T\big(|u_v(t)|^2 +
|v^{\e}(t)|_0^2 +1\big)\d t}.\label{29}
\end{align}
Since $v^{\e} \to v$ a.s. in the weak topology of $S_M$, it is clear
from the equation \eqref{29} that as $\e \to 0$,
\begin{align*}
E\Big[\sup_{0\leq t\leq T}|\w(t)|^2 + \nu\int_0^T \|\w(t)\|^2 \d
t\Big] \to 0.
\end{align*}
Let $\delta > 0$ be any arbitrary number. Then by Markov's inequality
\begin{align*}
&P\big\{\sup_{0\leq t\leq T}|\w(t)|^2 + \nu\int_0^T \|\w(t)\|^2 \d t
\geq \delta\big\}\nonumber\\
&\quad \leq \frac{1}{\delta} E\big[\sup_{0\leq t\leq T}|\w(t)|^2 +
\nu\int_0^T \|\w(t)\|^2 \d t\big] \to 0\ \text{as}\ \e \to 0.
\end{align*}
Thus
\begin{align*}
\sup_{0\leq t\leq T}|\uve(t)-u_v(t)|^2 + \nu\int_0^T
\|\uve(t)-u_v(t)\|^2 \d t \to 0
\end{align*}
in probability as $\e\to 0$. The proof is now complete.
\end{proof}

\begin{remark}
Sabra shell model of turbulence is the other well accepted model in the literature, and the fundamental difference with the GOY model lies in the number of complex conjugation operators used in the nonlinear terms which are
responsible for differences in the phase symmetries of the two
models, and as a consequence, Sabra shell  model exhibits
shorter-ranged correlations than the GOY model (see L'vov et. al.
\cite{Lv}). The equations of motion of the stochastic Sabra shell model have the following form
\begin{align*}
\frac{\d u_n}{\d t} + \nu k_n^2 u_n &+ i\big(a k_{n+1} u_{n+2}u\s_{n+1} + b k_n u_{n+1}u\s_{n-1} - \nonumber\\ & -ck_{n-1} u_{n-1}u_{n-2}\big)
= f_n + \sigma_n(t, u_n)\frac{\d w_n(t)}{\d t}, \quad\text{for}\ n= 1, 2, \ldots,
\end{align*}
along with the boundary conditions
\begin{equation*}
u_{-1} = u_0 = 0.
\end{equation*}
Under the same assumptions on the noise and noise coefficient given in Chapter $3$, and under the same functional setting, the existence and uniqueness of the strong solution can be established in $\mathrm{L}^2(\Omega; C(0, T; H)\cap\mathrm{L}^2(0, T; V))$. Moreover, by proceeding in the similar fashion as in Chapter $4$, one can easily verify the key estimates and prove the large deviation principle for the solution of the stochastic Sabra model in the Polish space $C(0, T; H)\cap\mathrm{L}^2(0, T; V)$.
\end{remark}

\medskip\noindent
{\bf Acknowledgements:} The first author would like to thank
Institut Mittag-Leffler (The Royal Swedish Academy of Sciences) for
their warm hospitality and support during the visit in
September--October 2007, where this work was initiated. He also
wants to thank Max-Planck Institute for Mathematics in the Sciences
in Leipzig, Germany for providing support and excellent research
environment which helped to complete this work. The second author
would like to thank the Army Research Office, Probability and
Statistics Program for their grant (DODARMY$41712$).

\end{document}